\documentclass{sig-alternate-2013}

\usepackage[utf8]{inputenc}

\usepackage{amsmath,amssymb}
\usepackage{amsthmnoproof}
\usepackage{amsrefs}
\usepackage[usenames,dvipsnames]{color}
\usepackage{stmaryrd}
\usepackage{enumerate}
\usepackage[algoruled,vlined,english,linesnumbered]{algorithm2e}
\usepackage[pdfpagelabels,colorlinks=true,citecolor=blue]{hyperref}
\usepackage{comment}
\usepackage{multirow}
\usepackage{tikz}

\newcommand{\noopsort}[1]{}
\DeclareMathOperator{\NP}{NP}
\DeclareMathOperator{\HP}{HP}
\DeclareMathOperator{\PP}{PP}
\DeclareMathOperator{\Hom}{Hom}
\DeclareMathOperator{\End}{End}
\DeclareMathOperator{\GL}{GL}
\DeclareMathOperator{\val}{val}

\DeclareMathOperator{\tr}{Tr}
\DeclareMathOperator{\com}{Com}
\DeclareMathOperator{\Grass}{Grass}

\DeclareMathOperator{\rank}{rank}

\newcommand{\Z}{\mathbb Z}

\newcommand{\Q}{\mathbb Q}
\newcommand{\Qp}{\Q_p}

\newcommand{\R}{\mathbb R}
\renewcommand{\O}{\mathcal O}
\newcommand{\OK}{\mathcal{O}_K}

\newcommand{\id}{\textrm{id}}

\renewcommand{\c}{\text{\rm c}}

\newcommand{\detp}{\det{'}}
\newcommand{\low}{\text{\rm low}}
\newcommand{\up}{\text{\rm up}}
\newcommand{\DI}{\text{\rm DI}}
\newcommand{\II}{\text{\rm II}}
\DeclareMathOperator{\charpoly}{char}
\newcommand{\charp}{\charpoly'}

\definecolor{purple}{rgb}{0.6,0,0.6}

\def\binom#1#2{\Big(\begin{array}{cc} #1 \\ #2 \end{array}\Big)}

\permission{%
Permission to make digital or hard copies of all or part of this work for
personal or classroom use is granted without fee provided that copies
are not made or distributed for profit or commercial advantage and
that copies bear this notice and the full citation on the first page.
Copyrights for components of this work owned by others than the
author(s) must be honored. Abstracting with credit is permitted. To
copy otherwise, or republish, to post on servers or to redistribute to
lists, requires prior specific permission and/or a fee.
Request permissions from Permissions@acm.org.}
\conferenceinfo{ISSAC'15,}{July 6--9, 2015, Bath, United Kingdom.}
\copyrightetc{Copyright \copyright~2015 ACM \the\acmcopyr}
\crdata{978-1-4503-3435-8/15/07\ ...\$15.00\\
DOI: http://dx.doi.org/10.1145/2755996.2756655}

\begin{document}

\newtheorem{theo}{Theorem}[section]
\newtheorem{lem}[theo]{Lemma}
\newtheorem{prop}[theo]{Proposition}
\newtheorem{cor}[theo]{Corollary}
\newtheorem{quest}[theo]{Question}
\newtheorem{conj}[theo]{Conjecture}
\theoremstyle{definition}
\newtheorem{rem}[theo]{Remark}
\newtheorem{ex}[theo]{Example}
\newtheorem{deftn}[theo]{Definition}

\title{p-Adic Stability In Linear Algebra}

\numberofauthors{3}
\author{
\alignauthor Xavier Caruso\\
  \affaddr{Universit\'e Rennes 1}\\
  \affaddr{\textsf{xavier.caruso@normalesup.org}}
\alignauthor David Roe \\
  \affaddr{University of British Columbia}\\
  \affaddr{\textsf{roed.math@gmail.com}}
\alignauthor Tristan Vaccon\\
  \affaddr{Universit\'e Rennes 1}\\
  \affaddr{\textsf{tristan.vaccon@univ-rennes1.fr}}
}

\maketitle

\begin{abstract}
Using the differential precision methods developed previously by the same authors,
we study the $p$-adic stability of standard operations on matrices and vector 
spaces. We demonstrate that lattice-based methods surpass naive methods in many
applications, such as matrix multiplication and sums and intersections of subspaces.
We also analyze determinants, characteristic polynomials and LU factorization using these differential methods.
We supplement our observations with numerical experiments.
\end{abstract}


%
%

\section{Introduction}

For about twenty years, the use of $p$-adic methods in symbolic 
computation has been gaining popularity. Such methods were 
used to compute composed products of polynomials 
\cite{boston-gonzalez-perdry-schost:05a}, 
to produce hyperelliptic curves of genus $2$ with complex multiplication 
\cite{gaudry-houtmann-weng-ritzenthaler-kohel:06a},
to compute isogenies between elliptic curves \cite{lercier-sirvent:08a} 
and to count points on varieties using $p$-adic 
cohomology theories (\emph{cf.} \cite{kedlaya:01a,lauder:04a} and many
followers).
However, a general framework allowing a precise study of $p$-adic 
precision --- the main issue encountered when computing 
with $p$-adic numbers --- was designed only recently in 
\cite{caruso-roe-vaccon:14a}. 

In \cite{caruso-roe-vaccon:14a}, we advocate the use of lattices
to track the precision of vectors and points on $p$-adic manifolds.
The main result of \emph{loc.~cit.}, recalled in Proposition \ref{prop:precision},
allows for the propagation of precision using differentials.  While
representing precision by lattices carries a high
cost in space and time requirements, the reduced precision loss
can sometimes overwhelm these costs within a larger algorithm.
In this paper, we apply the ideas of \cite{caruso-roe-vaccon:14a} to
certain linear algebraic tasks.

\smallskip

\noindent
{\bf Main Results.}
We give a number of contexts where lattice-based precision methods
outperform the standard coordinate-wise methods.  The two most striking
examples are an analysis of matrix multiplication and of 
sums and intersections of subspaces.  In Proposition \ref{prop:mulmatrix},
we give a formula for the lattice-precision of the product of two matrices.
In Figure \ref{fig:mulmatrix}, we describe the precision loss in multiplying
$n$ matrices, which appears linear in $n$ when using standard methods
and logarithmic in $n$ when using lattice methods.  We also give an
example in Figure \ref{fig:vectorspace} where lattice methods
actually yield an \emph{increase} in precision as the computation
proceeds.

\smallskip

\noindent
{\bf Organization of the paper.}
Section \ref{ssec:latticeprec} recalls the theory of 
precision developed in \cite{caruso-roe-vaccon:14a} and defines
a notion of diffuse precision for comparing to coordinate-wise methods.
In particular, we recall Proposition \ref{prop:precision}, which
allows the computation of precision using differentials.
Proposition \ref{prop:boundLambdaf2} in Section \ref{ssec:boundLambdaf}
is a technical result that will be used to determine the applicability of
Proposition \ref{prop:precision}.

In Section \ref{ssec:mulmatrix}, we analyze matrix multiplication and report
on experiments that demonstrate the utility of lattice-based precision tracking.
Section \ref{ssec:det} finds the precision of the determinant of a
matrix, and Section \ref{ssec:charpoly} applies the resulting formula to
characteristic polynomials.  We define the precision polygon of a matrix, which
gives a lower bound on the precision of the characteristic polynomial.
This polygon lies above the Hodge polygon of the matrix;
we give statistics on the difference and the amount of diffuse precision.
Finally, we apply Proposition \ref{prop:precision} to LU factorization and
describe experiments with lattice-based precision analysis.
Section \ref{ssec:grassgeo} reviews the geometry of Grassmannians,
which we apply in Section \ref{ssec:grassdiff} to differentiating
the direct image, inverse image, sum and intersection maps
between Grassmannians.  We then report in Section \ref{grassimpl}
on tracking the precision of subspace arithmetic in practice.
In the appendix, we give a proof of Proposition \ref{prop:boundLambdaf}.

The code used to make experiments presented in this paper is 
available at \url{https://github.com/CETHop/padicprec}.

\smallskip

\noindent
{\bf Notation.}
Throughout the paper, $K$ will refer to a complete discrete 
valuation field. Usual examples are finite extensions of $\Qp$ and 
Laurent series fields over a field. We denote by $\val : K \twoheadrightarrow \Z \cup 
\{+\infty\}$ the valuation on $K$, by 
$\OK$ the ring of integers and by $\pi \in K$ an element of valuation 
$1$. We let $\Vert \cdot \Vert$ be the norm associated to $\val$.

\medskip

\section{The theory of p-adic precision}
\label{sec:theory}

The aim of this section is to briefly summarize the content of 
\cite{caruso-roe-vaccon:14a} and fill in certain details.

\subsection{Lattices as precision data}
\label{ssec:latticeprec}

In \cite{caruso-roe-vaccon:14a}, we suggest the 
use of lattices to represent the precision of elements in 
$K$-vector spaces.  We shall contrast this approach with the
\emph{coordinate-wise method} used in Sage, 
where the precision of an element is specified by giving the precision
of each coordinate separately and is updated after each basic
operation.

Consider a finite 
dimensional\footnote{The framework of \cite{caruso-roe-vaccon:14a} is 
actually those of Banach spaces. However, we will not need infinite 
dimensional spaces here.} normed vector space $E$ 
defined over $K$. We use the notation $\Vert \cdot \Vert_E$ for the norm 
on $E$ and $B^-_E(r)$ (resp. $B^{\phantom -}_E(r)$) for the open (resp. 
closed) ball of radius $r$ centered at the origin. A \emph{lattice} $L \subset 
E$ is a sub-$\O_K$-module which generates $E$ over $K$. 
Since we are working in a ultrametric world, the balls $B^{\phantom 
-}_E(r)$ and $B^-_E(r)$ are examples of lattices. Actually, lattices 
should be thought of as special neighborhoods of $0$ and therefore are good 
candidates to model precision data. Moreover, as revealed in 
\cite{caruso-roe-vaccon:14a}, they behave quite well under (strictly) 
differentiable maps:

\begin{prop}
\label{prop:precision}
Let $E$ and $F$ be two finite dimensional normed vector spaces over $K$ 
and $f : U \rightarrow F$ be a function defined on an open subset $U$ of 
$E$. We assume that $f$ is differentiable at some point $v_0 \in U$ and 
that the differential $f'(v_0)$ is surjective.
Then, for all $\rho \in (0, 1]$, there exists a positive real
number $\delta$ such that, for all $r \in (0, \delta)$, any lattice
$H$ such that $B^-_E(\rho r) \subset H \subset B^{\phantom -}_E(r)$ 
satisfies:
\begin{equation}
\label{eq:firstorder}
f(v_0 + H) = f(v_0) + f'(v_0) (H).
\end{equation}
\end{prop}

This proposition enables the \emph{lattice method} of tracking precision,
where the precision of the input is specified as a lattice $H$ and precision
is tracked via differentials of the steps within a given algorithm.
The equality sign in Eq.~\eqref{eq:firstorder} shows that this method
yields the optimum possible precision. 
We refer to \cite[\S 4.1]{caruso-roe-vaccon:14a} for a more complete 
exposition.

In \cite{caruso-roe-vaccon:14a}, we also explained that if 
$f$ is locally analytic, then the constant $\delta$ appearing in
Proposition \ref{prop:precision} can be expressed in terms of the
growing function $\Lambda(f)$ defined by
\[
\textstyle \Lambda(f)(v) = 
\log \big( \sup_{h \in B^-_E(e^v)} \Vert 
f(h) \Vert \big)
\]
with the convention that $\Lambda(f)(v) = +\infty$ if $f$ does not
converge on $B^-_E(e^v)$.
We refer to \cite[Proposition 3.12]{caruso-roe-vaccon:14a} for the
precise statement.  We state here the case of
integral polynomial functions. A function $f : E \to F$ is said to be
\emph{integral polynomial} if it is given by multivariate polynomial 
functions with coefficients in $\O_K$ in any (equivalently all) system 
of coordinates associated to a $\OK$-basis of $B_E(1)$.

\begin{prop}
\label{prop:precision2}
We keep the notation of Proposition \ref{prop:precision} and assume 
in addition that $f$ is integral polynomial. Let $C$ be a positive real
number such that $B_F(1) \subset f'(v_0)(B_E(C))$. 
Then Proposition \ref{prop:precision} holds with $\delta = C \cdot
\rho^{-1}$.
\end{prop}

In \cite[Appendix A]{caruso-roe-vaccon:14a}, the theory is extended 
to manifolds over $K$, where the precision 
datum at some point $x$ is a lattice in the tangent space at $x$. 
Propositions \ref{prop:precision} and \ref{prop:precision2} have
analogues obtained by working in charts.  We use this extension
to compute with vector spaces in \S \ref{sec:vectorspaces}.

We will use the following definition in contrasting lattice and coordinate-wise methods. 
Suppose $E$ is equipped with a basis $(e_1, \ldots, e_n)$ and write $\pi_i : E \to Ke_i$ for the projections.

\begin{deftn} \label{def:diffused}
Let $H \subset E$ be a lattice. The number of 
\emph{diffused digits of precision} of $H$ is the length of
$H_0/H$ where $H_0 = \pi_1(H) \oplus \cdots \oplus \pi_n(H)$.
\end{deftn}
If $H$ represents the actual precision of some object, then
$H_0$ is the smallest diagonal lattice containing $H$.  Since
coordinate-wise methods cannot yield a precision better than $H_0$,
$k$ provides a lower bound on the number of $p$-adic digits
gained by lattice methods over standard methods.

\subsection{A bound on a growing function}
\label{ssec:boundLambdaf}

In the next sections, we will compute the derivative of several standard 
operations and sometimes give a simple expression in term 
of the input and the output. In other words, the function $f$ modeling
such an operation satisfies a differential equation of the form
$f' = g \circ (f, \id)$  
where $g$ is a given --- and hopefully rather simple --- function. The 
aim of this subsection is to study this differential equation and to 
derive from it certain bounds on the growing function $\Lambda(f)$. 
\emph{We will assume that $K$ has 
characteristic $0$.}

Let $E, F$ and $G$ be finite-dimensional normed vector spaces with
$U \subseteq E$ and $V \subseteq F$ and $W \subset G$ open subsets.
Generalizing the setting above, we consider the 
differential equation:
\begin{equation}
\label{eq:diffequah}
f' = g \circ (f, h).
\end{equation}
Here $g : V \times W \to \Hom(E, F)$ and $h : U \to W$ are known
locally analytic functions and $f : U \to V$ is the 
unknown locally analytic function. 
In what follows, we always assume that $V$ and $W$ contain the origin, 
$f(0) = 0$, $h(0) = 0$ and $g(0) \neq 0$. These assumptions are harmless 
for two reasons: first, we can always shift $f$ and $h$ (and $g$ 
accordingly) so that they both vanish at $0$, and second, in order to 
apply Proposition \ref{prop:precision2} the derivative $f'(0)$ 
needs to be surjective and therefore \emph{a fortiori} nonzero.

We assume that we are given in addition two nondecreasing convex 
functions $\Lambda_g$ and $\Lambda_h$ such that $\Lambda(g) \leq 
\Lambda_g$ and $\Lambda(h) \leq \Lambda_h$. We suppose further that 
there exists $\nu$ such that $\Lambda_g$ is constant on the interval 
$({-}\infty, \nu]$\footnote{We note that this assumption is fulfilled if 
we take $\Lambda_g = \Lambda(g)$ because we have assumed that $g(0)$ 
does not vanish.}. We introduce the functions $\tau_\nu$ and $\Lambda_f$ 
defined by:

\vspace{-0.2cm}

\begin{align*}
\tau_\nu(x) &= \smash{\left\{\begin{array}{ll} x & \mbox{ if } x \le \nu, \\
+\infty & \mbox{ otherwise;} \end{array}\right.} \\[8pt]
\mbox{and } \Lambda_f(x) &= \tau_\nu \circ (\id + \Lambda_g \circ \Lambda_h)(x + \alpha),
\end{align*}
where $\alpha$ is a real number satisfying $\Vert n! \Vert \geq 
e^{-\alpha n}$ for all $n$. If $p$ is the characteristic of the residue field,
a suitable value for $\alpha$ is $\alpha = - 
\frac p {p-1} \cdot \log \Vert p \Vert$ if $p > 0$ and $\alpha = 0$ if $p = 0$.
The next Proposition is proved in Appendix \ref{app:proof}.

\begin{prop}
\label{prop:boundLambdaf} \label{PROP:BOUNDLAMBDAF}
We have $\Lambda(f) \leq \Lambda_f$.
\end{prop}

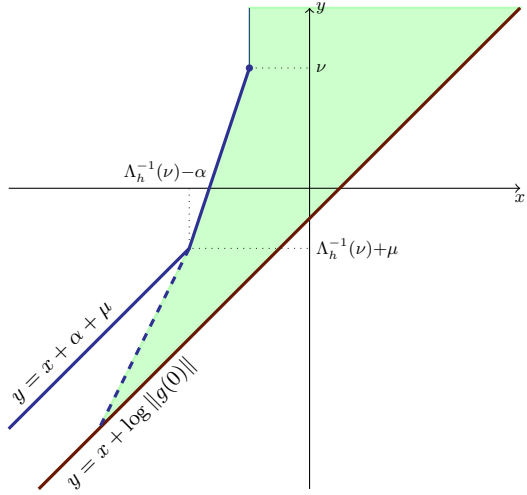
\begin{figure}
\null \hfill
\begin{tikzpicture}[scale=0.8]
\draw[thick,green!30,fill=green!20] 
   (-2,-1)--(-1,2)--(-1,3)--(3.5,3)
 --(2.5,2)--(-3.5,-4)--cycle;
\draw[->] (-5,0)--(3.5,0);
\draw[->] (0,-5)--(0,3);
\node[below, scale=0.75] at (3.5,0) { $x$ };
\node[right, scale=0.75] at (0,3) { $y$ };
\draw[black!80,dotted] (-1,2)--(0,2);
\node[right,scale=0.75] at (0,2) { $\nu$ };
\draw[black!80,dotted] (-2,0)--(-2,-1)--(0,-1);
\node[above,scale=0.75] at (-2.4,0) { $\Lambda_h^{-1}(\nu){-}\alpha$ };
\node[right,scale=0.75] at (0,-1) { $\Lambda_h^{-1}(\nu){+}\mu$ };
\draw[Blue,very thick] (-5,-4)--(-2,-1)--(-1,2);
\fill[Blue] (-1,2) circle (0.06);
\draw[Blue] (-1,2)--(-1,3);
\draw[Sepia, very thick] (-4.5,-5)--(3.5,3);
\draw[Blue, very thick,dashed] (-2,-1)--(-3.5,-4);
\node[below right,rotate=45,scale=0.9] at (-4.3,-4.8) 
  { $y = x + \log \Vert g(0) \Vert$ };
\node[above right,rotate=45,scale=0.9] at (-4.8,-3.8) 
  { $y = x + \alpha + \mu$ };
\end{tikzpicture}
\hfill \null

\vspace{-0.1cm}

\caption{Admissible region for the graph of $\Lambda(f)$}
\label{fig:area}

\end{figure}

Figure \ref{fig:area} illustrates Proposition \ref{prop:boundLambdaf}. 
The blue plain line represents the graph of the function $\Lambda_f$. A 
quick computation shows that, on a neighborhood of ${-}\infty$, this 
function is given by $\Lambda_f(x) = x + \alpha + \mu$
where $\mu$ is the value that $\Lambda_g$ takes on the interval 
$({-}\infty, \nu]$. Proposition \ref{prop:boundLambdaf} says that the
graph of $\Lambda(f)$ lies below the plain blue line. Moreover, we remark
that the Taylor expansion of $f(z)$ starts with the term
$g(0) z$. Hence, on a neighborhood of ${-}\infty$, we have 
$\Lambda(f)(x) = x + \log \Vert g(0) \Vert$. Using convexity, we 
get:
$\Lambda(f)(x) \geq x + \log \Vert g(0) \Vert$
for all $x \in \R$.  In other words, the graph of $\Lambda(f)$ lies above the brown line.
Furthermore, we know that the slopes of $\Lambda(f)$ are all integral
because $f$ is locally analytic. Hence, $\Lambda(f)$ cannot lie above
the dashed blue line defined as the line of slope $2$ passing through
the first break point of the blue plain line, which has coordinates
$(y_0 - \alpha - \mu, y_0)$ with $y_0 = \min(\Lambda_h^{-1}(\nu) + \mu, 
\nu)$. As a conclusion, we have proved that the graph of $\Lambda(f)$ 
must coincide with the brown line until it meets the dashed blue line 
and then has to stay in the green area.

As a consequence, we derive the following 
proposition, which can be combined with Proposition 3.12 of 
\cite{caruso-roe-vaccon:14a}. Following \cite{caruso-roe-vaccon:14a},
if $\varphi$ is a convex function and $v \in \R$, we define
$$\varphi_{\geq v} : x \mapsto \inf_{y \geq 0} \big( \varphi(x+y)
- vy \big).$$
It is the highest convex function with $\varphi_{\geq v} \leq \varphi$ and 
$\varphi'_{\geq v} \geq v$.

\begin{prop}
\label{prop:boundLambdaf2}
Keeping the above notation, we have:
$$\Lambda(f)_{\geq 2} (x) \leq 2(x + \alpha + \mu) -
\min(\Lambda_h^{-1}(\nu) + \mu, \: \nu)$$
for all $x \leq \min(\Lambda_h^{-1}(\nu) - \alpha, \: \nu - \mu - \alpha)$.
\end{prop}

\begin{proof}
The inequality follows from the fact that $y = 2(x + \alpha + \mu) - y_0$ is the equation of 
the dashed blue line.
\end{proof}

\begin{rem}
In certain situations,
it may happen that the function $f$ is solution of a simpler 
differential equation of the form $f' = g \circ f$. If this holds, 
Proposition \ref{prop:boundLambdaf2} gives the bound $\Lambda(f)_{\geq 
2} (x) \leq 2(x + \alpha + \mu) - \nu$ for $x \leq \nu - \mu - \alpha$.

Beyond this particular case, we recommend choosing the 
function $h$ by endowing $F$ with the second norm $\Vert x \Vert'_F = 
e^\mu \cdot \Vert x \Vert_F$ ($x \in F$) and taking $h : (F, \Vert 
\cdot \Vert_F) \to (F, \Vert \cdot \Vert'_F)$ to be the identity on the
underlying vector spaces. The function $\Lambda(h) : \R \to \R$ then maps $x$ to $x + 
\mu$ and we can choose $\Lambda_h = \Lambda(h)$.
\end{rem}

\section{Matrices}
\label{sec:matrices}

Let $M_{m,n}(K)$ denote the space of $m \times n$ matrices over $K$.  
We will repeatedly use the Smith decomposition for $M \in M_{m,n}(K)$,
which is $M = U_M \cdot \Delta_M \cdot V_M$
with $U_M$ and $V_M$ unimodular and $\Delta_M$ diagonal.  Write $\sigma_i(M)$
for the valuation of the $(i,i)$-th entry of $\Delta_M$, and by convention set
$\sigma_i(M) = +\infty$ if $i > \min(m,n)$.  Order the $\sigma_i(M)$ so that $\sigma_i(M) \le \sigma_{i+1}(M)$.

\subsection{Multiplication}
\label{ssec:mulmatrix}

To begin with, we want to study the behavior of the precision when 
performing a matrix multiplication. Let $r$, $s$ and $t$ be three 
positive integers and assume that we want to multiply a matrix $A \in 
M_{r,s}(K)$ by a matrix $B \in M_{s,t}(K)$. This operation is of course 
modeled by the integral polynomial function:
$$\begin{array}{rcl}
\mathcal P_{r,s,t} : \quad M_{r,s}(K) \times M_{s,t}(K) & \to & 
M_{r,t}(K) \smallskip \\
(A,B) & \mapsto & AB.
\end{array}$$
According to Proposition \ref{prop:precision}, the behavior of the precision when 
computing $AB$ is governed by $\mathcal P'_{r,s,t}(A,B)$, the linear mapping that takes a pair 
$(dA,dB)$ to $A \cdot dB + dA \cdot B$.

To fix ideas, let us assume from now that the entries of $A$ and $B$ all 
lie in $\O_K$ and are known at the same precision $O(\pi^N)$. In order 
to apply Propositions \ref{prop:precision} and \ref{prop:precision2}, we then need to compute the image 
of the standard lattice $\mathcal L_0 = M_{r,s}(\O_K) \times 
M_{s,t}(\O_K)$ under $\mathcal P'_{r,s,t}(A,B)$. It is of course 
contained in $M_{r,t}(\O_K)$; this reflects the obvious fact that each 
entry of the product $AB$ is also known with precision at least $O(\pi^N)$. 
Nonetheless, it may happen that the above inclusion is strict, meaning 
that we are \emph{gaining} precision in those cases.

Set $a_i = \sigma_i(A)$ and $b_i = \sigma_i(B)$, and define $M_{r,t}((a_i),(b_j))$ 
as the sublattice of $M_{r,t}(\O_K)$ consisting of matrices $M = (M_{i,j})$ 
such that $\val(M_{i,j}) \geq \min(a_i,b_j)$ for all $(i,j)$.  

\begin{prop}
\label{prop:mulmatrix}
With the above notation, we have
\[
\begin{array}{rl}
& \mathcal P'_{r,s,t}(A,B)(\mathcal L_0)
= U_A \cdot M_{r,t}((a_i),(b_j)) \cdot V_B \smallskip \\
\text{and} &
\text{\emph{length}}\Big(\frac{M_{r,t}(\O_K)}{\mathcal P'_{r,s,t}(A,B)(\mathcal L_0) }
\Big) =
\displaystyle 
\sum_{i,j} \min(a_i,b_j)
\end{array}
\]
\end{prop}

\begin{proof}
We write $A \cdot dB + dA \cdot B = U_A \cdot M \cdot V_B$ with
$$M = \Delta_A \cdot V_A \cdot dA \cdot V_B^{-1} 
+ U_A^{-1} \cdot dB \cdot U_B \cdot \Delta_B.$$
When $dA$ varies in $M_{a,b}(\O_K)$ so does $V_A \cdot dA \cdot V_B^{-1}$
and therefore the first summand in $M$ varies in the subspace of 
$M_{r,t}(\O_K)$ consisting of matrices whose entries on the $i$-th
row have valuation at least $a_i$. Arguing similarly for the second
summand, we deduce the first statement of the Proposition. The second
statement is now clear.
\end{proof}

From the perspective of precision, the second statement of Proposition 
\ref{prop:mulmatrix} means that the computation of $AB$ gains 
$\sum_{i,j} \min(a_i,b_j)$ significant digits in absolute 
precision\footnote{We
note that, on the other side, the valuation of the entries of $AB$ may 
increase, meaning that we are also losing some significant digits if we 
are reasoning in relative precision.} as soon as $N > \min(a_r, b_t)$ 
(\emph{cf.} Proposition \ref{prop:precision2}). However, many of these digits are 
diffused in the sense of Definition \ref{def:diffused}.
To change bases in order to make this increased precision visible
with coordinates, write 
$AB = U_A \cdot P \cdot V_B$ with
$P = \Delta_A \cdot V_A \cdot U_B \cdot \Delta_B$.
Tracking precision in the usual way, the $(i,j)$-th entry
of $P$ is known at precision $O(\pi^{N + \min(a_i,b_j)})$.  Multiplication by
$U_A$ and $V_B$ then diffuses the precision across the entries of $AB$.

\smallskip

We now consider the impact of tracking this diffuse precision.
Although the benefit is not substantial for
a single product of matrices, it accumulates as we multiply a large
number of matrices.  We illustrate this phenomenon with the following simple
example.

\noindent\hrulefill

\noindent {\bf Algorithm 1:} {\tt example\_product}

\noindent{\bf Input:} a list $(M_1, \ldots, M_n)$ of square matrices
of size $d$.

\smallskip

\noindent 1.\ Set $P$ to the identity matrix of size $d$

\noindent 2.\ {\bf for} $j=1,\dots,n$ {\bf do} {\bf compute} $P = P M_i$

\noindent 3.\ {\bf return} the top left entry of $P$

\vspace{-1ex}\noindent\hrulefill

\medskip

\noindent
Figure \ref{fig:mulmatrix} compares the number of significant digits in 
\emph{relative} precision we are losing on the output of Algorithm 1 
when we are using, on the one hand, a standard coordinate-wise track of 
precision and, on the other hand, a lattice-based method to handle 
precision.
\begin{figure}
\begin{center}
\renewcommand{\arraystretch}{1.2}
\begin{tabular}{|c|c|c|c|}
\hline
\multirow{2}{*}{\hspace{0.2cm}$d$\hspace{0.2cm}} & 
\multirow{2}{*}{\hspace{0.2cm}$n$\hspace{0.2cm}} & 
\multicolumn{2}{|c|}{Average loss of precision} \\
\cline{3-4}
& & Coord-wise method & Lattice method \\
\hline 
$2$ & $10$ & $\hphantom{00}2.8$ & $2.4$ \\
$2$ & $100$ & $\hphantom{0}16.7$ & $5.0$ \\
$2$ & $1000$ & $157.8$ & $7.9$ \\
\hline
$3$ & $10$ & $\hphantom{00}2.2$ & $1.9$ \\
$3$ & $100$ & $\hphantom{0}12.8$ & $4.0$ \\
$3$ & $1000$ & $122.5$ & $7.0$ \\
\hline
\end{tabular}

\smallskip

{\small
Results for a sample of $1000$ random inputs in $M_{d,d}(\Z_2)^n$}
\end{center}
\renewcommand{\arraystretch}{1}

\vspace{-0.3cm}

\caption{Average loss of precision in Algorithm 1}
\label{fig:mulmatrix}
\end{figure}
We observe that, in the first case, the number of lost digits seems to 
grow linearly with respect to the number of multiplications we are 
performing (that is $n$) whereas, in the second case, the growth seems 
to be only logarithmic. It would be nice to have a precise formulation 
(and proof) of this heuristic.

Note that multiplication of many random matrices plays a central
role in the theory of random walks on homogeneous spaces \cite{Benoist:2012}.
Better stability in computing such products helps estimate
Lyapunov exponents in that context.

\subsection{Determinant}
\label{ssec:det}

The computation of the differential of $\det : M_{n,n}(K) \to K$
is classical: at a point $M$ it is the linear map
\[
\detp(M) : dM \mapsto \tr(\com(M) \cdot dM),
\]
where $\com(M)$ stands for the comatrix of $M$, which is $\det(M) M^{-1}$ when $M$ is invertible. If $\rank(M) < n - 1$,
then $\detp(M)$ is not surjective and we cannot apply Proposition 
\ref{prop:precision}. Therefore, we suppose that $\rank(M) \ge n-1$ for the rest of this section.

As with matrix multiplication, we first determine the image of
the standard lattice $\mathcal L_0 = M_{n,n}(\O_K)$ under $\detp(M)$.

\begin{prop}
\label{prop:detmatrix}
Setting $v = \sigma_1(M) + \cdots + \sigma_{n-1}(M)$, we have
$\detp(M)(\mathcal L_0) = \pi^v \O_K$.
\end{prop}

\begin{proof}
From the description of $\detp(M)$, we see that it is enough to prove 
that the smallest valuation of an entry of $\com(M)$ is $v$
or, equivalently, that the ideal of $\O_K$ generated by all 
minors of $M$ of size $(n-1)$ is $\pi^v
\O_K$. But this ideal remains unchanged when we multiply $M$ 
on the left or on the right by a unimodular matrix. Thus we may 
assume that $M = \Delta_M$, and the result becomes clear.
\end{proof}

In terms of precision, Proposition \ref{prop:detmatrix} implies that if 
$M$ is given at flat precision $O(\pi^N)$ with $N > v$, then $\det(M)$ 
is known at precision $O(\pi^{N+v})$.  Thus we are gaining $v$  
digits in absolute precision or, equivalently, losing $\sigma_n(M)$ digits
of relative precision.
Furthermore, one may compute $\det(M)$ with this optimal precision by finding an approximate
Smith decomposition with $\Delta_M$ known at precision $O(\pi^N)$ and
multiplying its diagonal entries. 

\subsection{Characteristic polynomials} \label{ssec:charpoly}

Write $\charpoly : M_{n,n}(K) \to K[X]$ for the characteristic 
polynomial, and $K_{<n}[X] \subset K[X]$ for the subspace consisting of 
polynomials of degree less than $n$.  Then the differential of 
$\charpoly$ at a point $M$ is
\[
\charp(M) : dM \mapsto \tr(\com(X - M) \cdot dM).
\]
The image is the $K$-span of the entries of $\com(X{-}M)$, which is 
clearly contained within $K_{<n}[X]$.  In fact, the image will equal 
$K_{<n}[X]$ as long as $M$ does not have two Jordan blocks with the same 
generalized eigenvalue. For now on, we make this assumption.

Recall that the \emph{Newton polygon} $\NP(f)$ of a polynomial $f(X) = 
\sum_{i} a_i X^i$ is the lower convex hull of the points $(i, 
\val(a_i))$ and the \emph{Newton polygon} $\NP(M)$ of a matrix $M$ is 
$\NP(\charpoly(M))$.  The \emph{Hodge polygon} $\HP(M)$ of $M$ is the lower 
convex hull of the points $(i, \sum_{j = 1}^{n-i} \sigma_j(M))$. For 
any matrix $M$, the polygon $\NP(M)$ lies above $\HP(M)$ 
\cite{kedlaya:padicDiffEq}*{Thm. 4.3.11}.

Such polygons arise naturally in tracking the precision of polynomials 
\cite{caruso-roe-vaccon:14a}*{\S 4.2}.  Any such polygon $P$ yields a 
lattice $\mathcal{L}_P$ in $K_{<n}[X]$ consisting of polynomials whose 
Newton polygons lie above $P$.  This lattice is generated by monomials 
$a_iX^i$, where $\val(a_i)$ is the ceiling of the height of $P$ at $i$.  
These polygons are used in a coordinate-wise precision tracking for polynomial
arithmetic. We now introduce another polygon, bounded 
between $\NP(M)$ and $\HP(M)$, that will provide an estimate on the precision of $\charpoly(M)$.

\begin{deftn}
The \emph{precision polygon} $\PP(M)$ of $M$ is the lower convex hull of 
the Newton polygon of the entries of $\com(X{-}M)$.
\end{deftn}

It is clear from the definition $\mathcal L_{\PP(M)} \subset 
\charp(M)(\mathcal{L}_0)$ where $\mathcal{L}_0$ is the standard lattice 
$M_{n,n}(\OK)$. More precisely, $\PP(M)$ is the smallest polygon $P$ for 
which the inclusion $\mathcal L_P \subset \charp(M)(\mathcal{L}_0)$ holds. By 
Proposition \ref{prop:precision}, the precision polygon determines the minimal 
precision losses possible when encoding polynomial precision using polygons.

It turns out that the precision polygon is related to the Hodge and Newton 
polygons. If a polygon $P$ has vertices $(x_i, y_i)$, we let $T_n(P)$ be 
the translated polygon with vertices $(x_i - n, y_i)$.

\begin{prop} \label{prop:polygons}
The precision polygon $\PP(M)$ lies between $T_1(\HP(M))$
and $\NP(M)$.

Moreover, $\PP(M)$ and $T_1(\HP(M))$ meet at $0$ and $n{-}1$.
\end{prop}

\begin{proof}
The coefficients of $\charpoly(M)$ can be expressed as traces of 
exterior powers: the coefficient of $X^i$ is $\tr(\Lambda^i(M))$,
which is $\tr(\Lambda^i(U_M) \Lambda^i(\Delta_M) \Lambda^i(V_M))$.  
Computing $\Lambda^i(\Delta_M)$, we get the first statement of
the Proposition. For $i = 1$, we further find that $\PP(M)$ vanishes
at the abscissa $n{-}1$. By definition so does $T_1(\HP(M))$. The fact 
that $\PP(M)$ and $T_1(\HP(M))$ meet at abscissa $0$ follows from
Proposition \ref{prop:detmatrix}.

It remains to prove the comparison with the Newton polygon.
Set $f = \charpoly(M)$, set $m_{i,j}$ as the $(i,j)$-th entry of $M$, 
$f_{i,j}$ as the $(i,j)$-th entry of $\com(X{-}M)$ and $\mu_{i,j} = 
\val(m_{i,j})$. We write $f[k]$ 
for valuation of the coefficient of $X^k$ in $f$, and set $f[-1] = 
+\infty$. The equation $(X{-}M) \cdot \com(X{-}M) = f \cdot I$ yields, 
for all $i$ and $k$,
\begin{align*}
f[k] &\ge \inf(f_{i,i}[k{-}1], \mu_{i,0} + f_{0,i}[k], \ldots, 
\mu_{i,n} + f_{n,i}[k]) \\
&\ge \inf(f_{i,i}[k{-}1], f_{j,i}[k]),
\end{align*}
with the infimum over $j$. 
Taking lower convex hulls and noting that $\PP(M)$ is nonincreasing,
which follows from the comparison with the Hodge polygon, we get the 
result.
\end{proof}

\begin{rem}
Experiments actually support the following stronger result: $\PP(M)$ is 
bounded above by $T_1(\NP(M))$.
\end{rem}

For many matrices, $\PP(M) = T_1(\HP(M))$.  For random matrices over 
$\Z_2$, the $2$-adic precision polygon is equal to the Hodge polygon in 
$99.5\%$ of cases in dimension $4$, down to $99.1\%$ in dimension $8$.  
Over $\Z_3$, the fraction rises to $99.98\%$, with no clear dependence on 
dimension.
Empirically, $\PP(M)$ seems most likely to differ from $T_1(\HP(M))$ at 
$1$, corresponding to the precision of the linear term of the 
characteristic polynomial.

Of course, the precision lattice $\mathcal{E} = \charp(M)(\mathcal L_0)$ 
may contain diffuse precision which is not encapsulated in $\PP(M)$.
Diffuse precision arises in $11\%$ of cases in 
dimension $3$, up to $15\%$ of cases in dimension $8$.  This 
percentage increases as $\val(\det(M))$ increases, reaching $34\%$ in 
dimension $9$ for matrices constrained to have determinant with $2$-adic 
valuation $12$.
Moreover, one can construct examples with arbitrarily large amounts of diffuse precision.
Suppose the $\sigma_i(M)$ are large.  
Proposition \ref{prop:polygons} implies that $\mathcal{E}$ is 
contained within $\O_{K, <n}[X]$ with index at least $\sum_{i=1}^{n-1} 
\sigma_i(M)$.  The precision lattice of $1 + M$ is obtained from 
$\mathcal{E}$ via the transformation $X \mapsto 1 + X$, but $\PP(1 + M)$ 
is now flat with height $0$.

For randomly chosen matrices, approximating $\mathcal{E}$ 
using the Hodge polygon loses only a small amount of precision. However,
if the $\sigma_i(M)$'s are large or if $M$ is a translate of
such a matrix, using lattice precision can be very useful.

\subsection{LU factorization}

In this section, we denote by $\Vert \cdot \Vert$ the subordinate matrix norm on $M_n(K)$
and, given a positive real number $C$, we let $B(C)$ be the closed ball 
in $M_n(K)$ centered at the origin of radius $C$. We consider
the following subsets of $M_n(K)$:

\noindent $\bullet$
$O_n$ is the open subset of matrices whose principal 
minors do not vanish (we recall that the latest condition implies the
existence and the uniqueness of a LU factorization);

\noindent $\bullet$
$U_n$ is the sub-vector space of upper-triangular 
matrices;

\noindent $\bullet$
$L_n^0$ (resp. $L_n^u$) is the sub-affine space of 
nilpotent (resp. unipotent) lower-triangular matrices.

\smallskip

\noindent
{\bf Calculus and precision.}
We choose to normalize the LU factorization by requiring that $L$ is 
unipotent and denote by $D : O_n \to L_n^u \times U_n$ the function 
modeling this decomposition. The computation of the differential of $D$ 
has already been done in \cite[Appendix B]{caruso-roe-vaccon:14a}. For
$M \in \O_n$ with $D(M) = (L,U)$, the linear mapping $D'(M)$ is given 
by:
$$
dM \mapsto (L \cdot \low(dX), \up(dX) \cdot U)
\text{ with } dX = L^{-1} \cdot dM \cdot U^{-1}
$$
where $\low$ (resp. $\up$) denotes the canonical projection of $M_n(K)$ 
onto $L_n^0$ (resp. $U_n$). 
It is easily checked that $D'(M)$ is bijective with inverse given by
$(A,B) \mapsto AU+LB$.

We now want to apply Proposition \ref{prop:boundLambdaf2} in order to 
derive a concrete result on precision. \emph{We then assume that $K$ has 
characteristic $0$.} We pick $M_0 \in
O_n$ and write $D(M_0) = (L_0,U_0)$. We consider the translated
function $f$ taking $M$ to $D(M_0+M)-D(M_0)$. We then have $f(0) = 0$ 
and $f'(M) = D'(M_0+M)$. Using the explicit description of the inverse
of $D'(M_0+M)$, we find $B(1) \subset f'(0) \cdot B(C)$
for $C = \max(\Vert U_0 \Vert, \Vert L_0 \Vert)$. Moreover, $f$ satisfies 
the differential equation $f' = g \circ f$ where $g$ is defined by:
\begin{equation}
\label{eq:gAB}
\begin{array}{l}
g(A,B)(X) = \big( (L_0+A) \cdot \low(Y), \up(Y) \cdot (U_0+B) \big)
\smallskip \\
\hspace{1.8cm}\text{with } Y = (L_0+A)^{-1} \cdot X \cdot (U_0+B)^{-1}.
\end{array}
\end{equation}
Let $\kappa(S) = \Vert S \Vert \cdot \Vert S^{-1} \Vert$ denote the 
condition number of a matrix $S$. Remarking that $\Vert S + T 
\Vert = \Vert S \Vert$ if $\Vert T \Vert < \Vert S \Vert$ and $\Vert (S 
+ T)^{-1} \Vert = \Vert S^{-1} \Vert$ if $\Vert T \Vert < \Vert S^{-1} 
\Vert^{-1}$, we deduce from \eqref{eq:gAB} that
$\Vert g(A,B) \Vert \leq \max 
\big( \kappa(L_0) \Vert U_0^{-1} \Vert, 
\kappa(U_0) \Vert L_0^{-1} \Vert \big)$
provided that $\Vert A \Vert < \Vert L_0^{-1} \Vert^{-1}$ and 
$\Vert B \Vert < \Vert U_0^{-1} \Vert^{-1}$.
Combining this with Proposition \ref{prop:boundLambdaf2} and 
\cite[Proposition.~3.12]{caruso-roe-vaccon:14a}, we finally find that 
Eq.~\eqref{eq:firstorder} holds as soon as 
$$\begin{array}{l}
\displaystyle \frac \rho r > \Vert p \Vert^{-\frac{2p}{p-1}} \cdot
\max ( \Vert L_0 \Vert, \Vert U_0 \Vert) \cdot
\max ( \Vert L_0 ^{-1} \Vert, \Vert U_0 ^{-1} \Vert) \cdot{} \\
\hspace{3.5cm} \max \big( \kappa(L_0) \Vert U_0^{-1} \Vert, 
  \kappa(U_0) \Vert L_0^{-1} \Vert \big)^2.
\end{array}$$

\noindent
{\bf Numerical experiments.}
Let $B_n=(E_{i,j})_{1 \leq i,j\leq n}$ be the canonical basis of 
$M_n(K)$. It can be naturally seen as a basis of $L_n^0 \times U_n$ as
well. For a given $M \in O_n$, let us abuse notation and write $D'(M)$
for the matrix of this linear mapping in the above basis. We remark that
$D'(M)$ is lower-triangular in this basis. Projecting $D'(M)$ onto each
coordinate, we find the best coordinate-wise loss of precision we can 
hope for the computation of $D$ is given by $\sum_u \left( \max_v \left( 
\val(D'(M)_{u,v}) \right) \right)$. This number should be compared to
$\val(\det(D'(M)))$, which is precision lost in the lattice method. 
The number of diffused digits of precision of $D'(M)(M_{n,n}(\OK))$ is 
then the difference between these two numbers.
Figure \ref{fig:LU} summarizes the mean and standard deviation of those 
losses for a sample of 2000 random matrices in $M_{d,d}(\mathbb{Z}_2)$.
\begin{figure}
\begin{center}
\renewcommand{\arraystretch}{1.2}
\begin{tabular}{|c|c|c|c|c|}
\hline 
& \multicolumn{4}{|c|}{Loss of precision in LU decomposition} \\
\cline{2-5}
\raisebox{0.2em}{matrix}
& \multicolumn{2}{|c|}{coord-wise method} & \multicolumn{2}{|c|}{lattice method}  \\  \cline{2-5}
\smash{\raisebox{0.6em}{size}} 
& \hspace{0.5em}mean\hspace{0.5em} & deviation 
& \hspace{0.5em}mean\hspace{0.5em} & deviation \\ \hline
$2$ & $\phantom{2}3.0$& $5$ & $1.5$ & $1.4$ \\
$3$& $\phantom{2}9.4$& $11$ & $2.3$ & $2.3$\\
$4$ & $20\phantom{.4}$& $20$ & $3.8$ & $3.1$\\ \hline
\end{tabular} 
\smallskip

{\small
Results for a sample of $2000$ instances}
\end{center}
\renewcommand{\arraystretch}{1}

\vspace{-0.3cm}

\caption{Loss of precision for LU factorization}
\label{fig:LU}

\end{figure}
%

%

\section{Vector spaces}
\label{sec:vectorspaces}

Vector spaces are generally represented as subspaces of $K^n$ for some 
$n$ and hence naturally appear as points on Grassmannians. Therefore, 
one can use the framework of \cite[Appendix A]{caruso-roe-vaccon:14a} to 
study $p$-adic precision in this context.

\subsection{Geometry of Grassmannians}
\label{ssec:grassgeo}

Given $E$, a finite dimensional vector space over $K$, and $d$, an 
integer in the range $[0, \dim E]$, we write $\Grass(E,d)$ for the
Grassmannian of $d$-dimensional subspaces of $E$. It is well-known
that $\Grass(E,d)$ has the natural structure of a $K$-manifold. The aim of
this subsection is to recall standard facts about its geometry. In what follows,
we set $n = \dim E$ and equip $E$ with a distinguished basis
$(e_1, \ldots, e_n)$.

\smallskip

\noindent
{\bf Description and tangent space.}
Let $V$ denote a fixed subspace of $E$ of dimension $d$. The 
Grassmannian $\Grass(E,d)$ can be viewed as the quotient of the set 
of linear embeddings $f: V \hookrightarrow E$ modulo the action (by 
precomposition) of $\GL(V)$: the mapping $f$ represents its image 
$f(V)$. It follows from this description that the tangent space of 
$\Grass(E,d)$ is canonically isomorphic to $\Hom(V, E) / \End(V)$,
which is $\Hom(V, E/V)$.

\smallskip

\noindent
{\bf Charts.}
Let $V$ and $V^\c$ be two complementary subspaces of $E$ 
(\emph{i.e.} $V \oplus V^\c = E$). We assume that $V$ has 
dimension $d$ and denote by $\pi$ the projection $E \to V$ 
corresponding to the above decomposition. We introduce the set 
$\mathcal U_{V,V^\c}$ of all embeddings $f : V \hookrightarrow E$ 
such that $\pi \circ f = \id_V$. Clearly, it is an affine space over
$\Hom(V,V^\c)$. 
Furthermore, we can embed it into $\Grass(E,d)$ by taking $f$ as
above to its image. This way, $\mathcal U_{V,V^\c}$ appears as
an open subset of $\Grass(E,d)$ consisting exactly of those subspaces 
$W$ such that $W \cap V^\c = 0$. As a consequence, the tangent space 
at each such $W$ becomes isomorphic to $\Hom(V,V^\c)$. The
identification $\Hom(V,V^\c) \to \Hom(W, E/W)$ is given by
$du \mapsto (du \circ f^{-1}) \text{ mod } W$ where $f : V 
\stackrel{\sim}{\to} W$ is the linear mapping defining $W$.

When the pair $(V, V^\c)$ varies, the open subsets $\mathcal 
U_{V,V^\c}$ cover the whole Grassmannian and define an atlas.
When implementing vector spaces on a computer, we usually restrict 
ourselves to the subatlas consisting of all charts of the form $(V_I, 
V_{I^\c})$ where $I$ runs over the family of subsets of $\{1, 
\ldots, n\}$ of cardinality $d$ and $V_I$ is the subspace spanned by 
the $e_i$'s with $i \in I$. A subspace $W \in E$ then belongs to at 
least one $\mathcal U_{V_I, V_{I^{\text{c}}}}$ and, given a family of
generators of $W$, we can determine such an $I$ together with the
corresponding embedding $f : V_I \hookrightarrow E$ by row reducing 
the matrix of generators of $W$.

\smallskip

\noindent
{\bf A variant.}
Alternatively, one can describe $\Grass(E,d)$ as the set of linear 
surjective morphisms $f : E \to E/V$ modulo the action (by 
postcomposition) of $\GL(E/V)$. This identification presents the 
tangent space at a given point $V$ as the quotient $\Hom(E, E/V) / 
\End(E/V) \simeq \Hom(V, E/V)$.
Given a decomposition $E = V \oplus V^\c$ as above, we let $\mathcal 
U^\star_{V, V^\c}$ denote the set of surjective linear maps $f : E 
\to V^\c$ whose restriction to $V^\c$ is the identity. It is an 
affine space over $\Hom(V, V^\c)$ which can be identified with an open
subset of $\Grass(E,d)$ \emph{via} the map $f \mapsto \ker f$.

It is easily seen that $\mathcal U_{V, V^\c}$ and $\mathcal 
U^\star_{V, V^\c}$ define the same open subset in $\Grass(E,d)$. 
Indeed, given $f \in \mathcal U_{V, V^\c}$, one can write $f = \id_V 
+ h$ with $h \in \Hom(V, V^\c)$ and define the morphism $g = \id_E -
h \circ \pi \in \mathcal U^\star_{V, V^\c}$. The association $f \mapsto
g$ then defines a bijection $\mathcal U_{V, V^\c} \to \mathcal
U^\star_{V, V^\c}$ which commutes with the embeddings into the
Grassmannian.

\smallskip

\noindent
{\bf Duality.}
If $E$ is a finite dimensional vector space over $K$, we use the notation 
$E^\star$ for its dual (\emph{i.e.} $E^\star = \Hom(E,K)$). If we are also given
a subspace $V \subset E$, we denote by $V^\perp$ the subspace 
of $E^\star$ consisting of linear maps that vanish on $V$. We recall
that the dual of $V^\perp$ (resp. $E^\star/V^\perp$) is canonically
isomorphic to $E/V$ (resp. $V$).
For all $d$, the association $V \mapsto V^\perp$ defines a continuous 
morphism $\psi_E : \Grass(E,d) \to \Grass(E^\star, n-d)$. The action of 
$\psi_E$ on tangent spaces is easily described. Indeed, the differential 
of $\psi_E$ at $V$ is nothing but the canonical identification between
$\Hom(V, E/V)$ and $\Hom(V^\perp, E^\star/V^\perp)$ induced by 
transposition. Furthermore, we observe that $\psi_E$ respects
the charts we have defined above, in the sense that it maps bijectively 
$\mathcal U_{V, V^\c}$ to $\mathcal U^\star_{V^\perp, (V^\c)^\perp}
\simeq \mathcal U_{V^\perp, (V^\c)^\perp}$.

\subsection{Differential computations}
\label{ssec:grassdiff}

In this subsection, we compute the differential of various operations on 
vector spaces. For brevity, we skip the estimation of the corresponding 
growing functions (but this can be done using Proposition 
\ref{prop:boundLambdaf2} as before if $\text{char}(K) = 0$).

\smallskip

\noindent
{\bf Direct images.}
Let $E$ and $F$ be two finite dimensional $K$-vector spaces of dimension 
$n$ and $m$, respectively. Let $d$ be an integer in $[0,n]$. We are
interested in the direct image function $\DI$ defined 
on $\mathcal M = \Hom(E,F) \times \Grass(E,d)$ that takes the 
pair $(f,V)$ to $f(V)$. Since the dimension of $f(V)$ may vary,
the map $\DI$ does not take its values in a well-defined 
Grassmannian. We therefore stratify $\mathcal M$ as
follows: for each integer $r \in [0,d]$, let $\mathcal M_r \subset \mathcal M$
be the subset of pairs $(f,V)$ for which $f(V)$ has dimension $r$.
The $\mathcal M_r$'s are locally closed in $\mathcal M$ and are therefore
submanifolds. Moreover, $\DI$ induces differentiable 
functions $\DI_r : \mathcal M_r \to \Grass(F,r)$.

We would like to differentiate $\DI_r$ around some point $(f,V) \in 
\mathcal M_r$. To do so, we use the first description of the 
Grassmannians we gave above: we see points in $\Grass(E,d)$ 
(resp. $\Grass(F,d)$) as embeddings $V \hookrightarrow E$ (resp. $W 
\hookrightarrow F$) modulo the action of $\GL(V)$ (resp. $\GL(W)$).
The point $V \in \Grass(E,d)$ is then represented by the canonical 
inclusion $v : V \to E$ whereas a representative $w$ of $W$ satisfies
$w \circ \varphi = f \circ v$
where $\varphi : V \to W$ is the linear mapping induced by $f$. The 
previous relation still holds if $(f,v)$ is replaced by a pair $(f', 
v') \in \mathcal M_r$ sufficiently close to $(f,v)$.
Differentiating it and passing to the quotient we find, first, that the 
tangent space of $\mathcal M_r$ at $(f,v)$ consists of pairs $(df, dv) 
\in \Hom(E,F) \times \Hom(V,E/V)$ such that
$$d\tilde w = \big((df \circ v + f \circ dv) \text{ mod } W\big)
\in \Hom(V, F/W)$$
factors through $\varphi$ (\emph{i.e.} vanishes on $\ker \varphi = V 
\cap \ker f$) and, second, that the differential of $\DI_r$ at $(f,V)$ 
is the linear mapping sending $(df,dv)$ as above to the unique element 
$dw \in \Hom(W,F/W)$ such that $dw \circ \varphi = d \tilde w$.

\smallskip

\noindent
{\bf Inverse images.}
We now consider the inverse image mapping $\II$ sending a 
pair $(f,W) \in \mathcal W = \Hom(E,F) \times \Grass(F,d)$ to 
$f^{-1}(W)$. As before, this map does not take values in a single 
Grassmannian, so we need to stratify $\mathcal W$ in order to get 
differentiable functions. For each integer $s \in [0,n]$, we introduce 
the submanifold $\mathcal W_s$ of $\mathcal W$ consisting of those 
pairs $(f,W)$ such that $\dim f^{-1}(W) = s$. For all $s$, $\II$ 
induces a continuous function
$\II_s : \mathcal W_s \to \Grass(E,s)$.
Pick $(f,W) \in \mathcal W_s$. Set $V = f^{-1}(W)$ and denote by $w : F 
\to F/W$ the canonical projection.
Similarly to what we have done for direct images, one can prove that
the tangent space of $\mathcal W_s$ at some point $(f,W) \in \mathcal
W_s$ is the subspace of $\Hom(E,F) \times \Hom(W,F/W)$ consisting of 
pairs $(df,dw)$ such that
$d \tilde v = (w \circ df + dw \circ f)_{|W}$
factors through the linear mapping $\varphi : E/V \to F/W$ induced by 
$f$. Furthermore $\II_s$ is differentiable at $(f,W)$ and its 
differential is the linear mapping that takes $(df,dw)$ as above to the 
unique element $dv \in \Hom(V,E/V)$ satisfying $\varphi \circ dv =
d\tilde v$.

Direct images and inverse images are related by duality
as follows: if $f : E \to F$ is any linear map and $W$ is a subspace
of $F$, then $f^\star(W^\perp) = f^{-1}(W)^\perp$. We 
can thus deduce the differentials of $\DI_s$ from those of $\II_s$ and 
\emph{vice versa}.


\smallskip

\noindent
{\bf Sums and intersections.}
Let $d_1$ and $d_2$ be two nonnegative integers. We consider the 
function $\Sigma$ defined on the manifold $\mathcal C = \Grass(E,d_1) 
\times \Grass(E,d_2)$ by $\Sigma(V_1, V_2) = V_1 + V_2$. As before, in 
order to study $\Sigma$, we stratify $\mathcal C$ according to the 
dimension of the sum: for each integer $d \in [0, d_1+d_2]$, we define 
$\mathcal C_d$ as the submanifold of $\mathcal C$ consisting of those 
pairs $(V_1, V_2)$ such that $\dim(V_1 + V_2) = d$. We get a 
well-defined mapping $\mathcal C_d \to \Grass(E,d)$ whose differential
can be computed as before. The
tangent space of $\mathcal C_d$ at a given point $(V_1, V_2)$ consists 
of pairs $(dv_1, dv_2) \in \Hom(V_1, E/V_1)
\times \Hom(V_2, E/V_2)$ such that $dv_1 \equiv dv_2 \pmod{V_1 + V_2}$ 
on the intersection $V_1 \cap V_2$, and the differential of $\Sigma$
at $(V_1, V_2)$ maps $(dv_1, dv_2)$ to $dv \in \Hom(V, E/V)$ (with $V
= V_1 + V_2$) defined by $dv(v_1 + v_2) = dv_1(v_1) + dv_2(v_2)$ ($v_1
\in V_1$, $v_2 \in V_2$).

Using duality, we derive a similar result for the mapping 
$(V_1, V_2) \mapsto V_1 \cap V_2$ (left to the reader).

\subsection{Implementation and experiments}
\label{grassimpl}

\smallskip

\noindent
{\bf Standard representation of vector spaces.}  
One commonly represents subspaces of $K^n$  
using the charts $\mathcal U_{V_I, V_{I^\c}}$ (where $I$ is a subset 
of $\{1, \ldots, n\}$) introduced above. More concretely,
a subspace $V \subset K^n$ is represented as the span of the rows
of a matrix $G_V$ having the following extra property: there 
exists some $I \subset \{1, \ldots, n\}$ such that the submatrix of 
$G_V$ obtained by keeping only columns with index in $I$ is the identity 
matrix. We recall that such a representation always exists and, when 
the set of indices $I$ is fixed, at most one $G_V$ satisfies the above condition.
Given a family of generators of $V$, one can compute $G_V$ and $I$ as 
above by performing standard row reduction.  Choosing the first non-vanishing
pivot at every stage provides a canonical choice for $I$, but in the context of
inexact base fields, choosing the pivot with the maximal norm yields
a more stable algorithm.

\smallskip

\noindent
{\bf The dual representation.}
Of course, one may alternatively use the charts $\mathcal U^\star_{V_I, 
V_{I^\c}}$. Concretely, this means that we represent $V$ as the left 
kernel of a matrix $H_V$ having the following extra property: there 
exists some $I \subset \{1, \ldots, n\}$ such that the submatrix of 
$H_V$ obtained by \emph{deleting} rows with index in $I$ is the identity 
matrix. As above, we can then compute $I$ and $H_V$ by performing 
column reduction.

Note that switching representations is cheap and stable.
If $I = \{1, \ldots, d\}$ with $d = \dim V$ and $I_d$ 
is the identity matrix of size $d$, the matrix $G_V$ has the form
$(\begin{matrix} I_d & G'_V \end{matrix})$.
One can represent $V$ with the same $I$ and the matrix
$H_V = \Big(\begin{matrix} -G'_V \\ I_{n-d} \end{matrix}\Big)$.
A similar formula exists for general $I$.

\medskip

\noindent
{\bf Operations on vector spaces.}
The first representation we gave is well suited for the computation
of direct images and sums. For instance, to compute $f(V)$
we apply $f$ to each row of $G_V$, obtaining a
family of generators of $f(V)$, and then row reduce. Dually, the
second representation works well for computing inverse images, including 
kernels, and intersections. Since translating between the two dual
representations is straightforward, we get algorithms for solving both
problems using either representation.


\smallskip

\noindent
{\bf Some experiments.}
Let us consider the example computation given in the following
algorithm.

\noindent\hrulefill

\noindent {\bf Algorithm 2:} {\tt example\_vector\_space}

\noindent{\bf Input:} two integers $n$ and $N$

\smallskip

\noindent 1.\ Set $L_0 = \langle (1 + O(2^N),  O(2^N), O(2^N)) \rangle
\subset \Q_2^3$

\noindent 2.\ {\bf for} $i=0,\dots,n-1$

\noindent 3.\ \hspace{0.3cm}pick randomly $\alpha, \beta, \gamma, \delta
\in M_{3,3}(\Z_2)$ with high precision

\noindent 4.\ \hspace{0.3cm}{\bf compute} $L_{i+1} = 
\big(\alpha(L_i) + \beta(L_i)\big) \cap \big(\gamma(L_i) + \delta(L_i)\big)$

\noindent 5.\ {\bf return} $L_n$

\vspace{-1ex}\noindent\hrulefill

\medskip

\noindent
The expression \emph{high precision} on line 3 means that the precision on 
$\alpha$, $\beta$, $\gamma$ and $\delta$ is set in such a way that it 
does not affect the resulting precision on $L_{i+1}$.
\begin{figure}
\begin{center}
\renewcommand{\arraystretch}{1.2}
\begin{tabular}{|c|c|c|c|}
\hline
\multirow{3}{*}{\hspace{0.2cm}$n$\hspace{0.2cm}} & 
\multicolumn{3}{|c|}{Average loss of precision} \\
\cline{2-4}
& \raisebox{-0.05cm}{Coord-wise} & 
\multicolumn{2}{|c|}{Lattice method} \\
\cline{3-4}
& \raisebox{0.04cm}{method} & 
\hspace{0.2cm}Projected\hspace{0.2cm} & 
\hspace{0.2cm}Diffused\hspace{0.2cm} \\
\hline 
$10$ & $\hphantom{0}7.3$ & $\hphantom{0}2.7$ & $\hphantom{0}{-}2.4 \times 2$ \\
$20$ & $14.8$ & $\hphantom{0}5.5$ & $\hphantom{0}{-}4.7 \times 2$ \\
$50$ & $38.6$ & $13.1$ & $-12.0 \times 2$ \\
$100$ & $78.1$ & $26.5$ & $-23.5 \times 2$ \\
\hline
\end{tabular}

\smallskip

{\small
Results for a sample of $1000$ executions (with $N \gg n$)}
\end{center}
\renewcommand{\arraystretch}{1}

\vspace{-0.4cm}

\caption{Average loss of precision in Algorithm 2}
\label{fig:vectorspace}

\vspace{-0.2cm}

\end{figure}
Figure \ref{fig:vectorspace} shows the losses of precision when
executing Algorithm 2 with various inputs $n$ (the input $N$ is
always chosen sufficiently large so that it does not affect the
behavior of precision). The \emph{Coord-wise} column corresponds to
the standard way of tracking precision. On the other hand, in the two 
last columns, the precision is tracked using lattices. The \emph{Diffused}
column gives the amount of diffused precision, factored to be comparible
to the Coord-wise column. 
The fact that only negative numbers appear in this column means that we 
are actually always gaining precision with this model! Finally, the 
\emph{Projected} column gives the precision loss after projecting the lattice
precision onto coordinates.

\appendix

\section{Proof of Proposition \ref{prop:boundLambdaf}}
\label{app:proof}

We prove Proposition \ref{prop:boundLambdaf} in the slightly more
general context of $K$-Banach spaces.

\subsection{Composite of locally analytic functions}

Let $U$, $V$ and $W$ be three open subsets in $K$-Banach spaces 
$E$, $F$ and $G$, respectively. We assume that $0 \in U$, $0 \in V$. Let 
$f : U \to V$ and $g : V \to W$ be two locally analytic functions around 
$0$ with $f(0) = 0$. The composition $h = g \circ f$ is then locally 
analytic around $0$ as well. Let $f = \sum_{n \geq 0} f_n$
$g = \sum_{n \geq 0} g_n$ and $h = \sum_{n \geq 0} h_n$ be the
analytic expansions of $f$, $g$ and $h$.
Here $f_n$, $g_n$ and $h_n$ are the restrictions to the diagonal of 
some symmetric $n$-linear forms $F_n$, $G_n$ and $H_n$, respectively. The 
aim of this subsection is to prove the following intermediate result.

\begin{prop}
\label{prop:boundhr}
With the above notation, we have
$$\Vert h_r \Vert \leq \sup_{m, (n_i)}
  \Vert g_m \Vert \cdot \Vert f_{n_1} \Vert \cdots \Vert f_{n_m} \Vert$$
for all nonnegative integers $r$, where the supremum is taken over all pairs $(m, (n_i))$ where $m$
is a nonnegative integer and $(n_i)_{1 \leq i \leq m}$ is a sequence of
length $m$ of nonnegative integers such that $n_1 + \ldots + n_m = r$.
\end{prop}

A computation gives the following expansion for $g \circ f$:
\begin{equation}
\label{eq:expansiongf}
\sum \binom m {\!k_1 \,\, \cdots \,\, k_\ell\!} \:
G_m(f_{n_1}, \ldots, f_{n_1}, \ldots, f_{n_\ell}, \ldots, f_{n_\ell})
\end{equation}
where $\binom m {\!k_1 \,\, \cdots \,\, k_\ell\!}$ denotes the
multinomial coefficient and the sum runs over:

\noindent
(a) all finite sequences $(k_i)$ of positive integers whose length 
(resp. sum) is denoted by $\ell$ (resp. $m$), and

\noindent
(b) all finite sequences $(n_i)$ of positive integers of length
$\ell$.

\noindent
Moreover, in the argument of $G_m$, the variable $f_{n_i}$ 
is repeated $k_i$ times.

The degree of $G_m(f_{n_1}, \ldots, f_{n_1}, \ldots, f_{n_\ell},
\ldots, f_{n_\ell})$ is $r = k_1 n_1 + \ldots + k_\ell n_\ell$ and 
then contributes to $h_r$. As a consequence $h_r$ is equal to 
\eqref{eq:expansiongf} where the sum is restricted to sequences
$(k_i)$, $(n_i)$ such that $k_1 n_1 + \ldots + k_\ell n_\ell = r$.
Proposition \ref{prop:boundhr} now follows from the next lemma.

\begin{lem}
Let $E$ be a $K$-vector space. Let $\varphi : E^m \to 
K$ be a symmetric $m$-linear form and $\psi: E \to K$ defined by 
$\psi(x) = \varphi(x, x, \ldots, x)$.
Given positive integers $k_1, \ldots, k_\ell$ whose sum is $m$ and 
$x_1, \ldots, x_\ell \in E$, we have
$$\begin{array}{l}
\Big\Vert \binom m {k_1 \,\, k_2 \,\, \cdots \,\, k_\ell} \cdot
\varphi(x_1, \ldots, x_1, \ldots, x_\ell, \ldots,
x_\ell) \Big\Vert  \\
\hspace{4.3cm} \leq \Vert \psi \Vert \cdot \Vert x_1 \Vert^{k_1} \cdots
 \Vert x_\ell \Vert^{k_\ell}
\end{array}$$
where, in the LHS, the variable $x_i$ is repeated $k_i$ times.
\end{lem}

\begin{proof}
It is enough to prove that
$$\Big\Vert \binom m {k_1 \,\, k_2 \,\, \cdots \,\, k_\ell} \cdot
\varphi(x_1, \ldots, x_1, \ldots, x_\ell, \ldots,
x_\ell) \Big\Vert \leq \Vert \psi \Vert$$
provided that all the $x_i$'s have norm at most $1$. We proceed by 
induction on 
$\ell$. The $\ell = 1$ case follows directly from
the definition of $\Vert \psi \Vert$. 
We now pick $(\ell+1)$ integers $k_1, \ldots, k_{\ell+1}$ whose sum 
equals $m$, together with $(\ell+1)$ elements $x_1, \ldots, x_{\ell+1}$
lying in the unit ball of $E$.
We also consider a new variable $\lambda$ varying in
$\O_K$. We set $x'_i = x_i$, $k'_i = k_i$ when $i < \ell$ and $x'_\ell 
= x_\ell + \lambda x_{\ell+1}$ and $k'_\ell = k_\ell + k_{\ell+1}$. By the 
induction hypothesis, we know that the inequality
$$\Big\Vert \binom m {k'_1 \,\, \cdots \,\, k'_\ell} \cdot
\varphi(x'_1, \ldots, x'_1, \ldots, x'_\ell, \ldots, x'_\ell) \Big\Vert
\leq \Vert \psi \Vert$$
holds for all $\lambda \in K$. Furthermore, the LHS of the inequality
is a polynomial $P(\lambda)$ of degree $k'_\ell$ whose coefficient in 
$\lambda^j$ is
$$\binom m {k'_1 \,\, \cdots \,\, k'_\ell} \cdot
\binom {k'_\ell} {j} \cdot
\varphi(\underline x_j) = 
\binom m {k_1 \,\, \cdots \,\, k_{\ell-1} \,\, j} \cdot
\varphi(\underline x_j)$$
with
$\underline x_j = (x_1, \ldots, x_1, \ldots, x_{\ell+1}, \ldots, 
x_{\ell+1})$
where $x_i$ is repeated $k_i$ times if $i < \ell$ and $x_\ell$ 
(resp. $x_{\ell+1}$) is repeated $j$ times (resp. $k'_\ell - j$ times).
Since $\Vert P(\lambda) \Vert \leq \Vert \psi \Vert$ for all $\lambda$ in 
the unit ball, the norm of all its coefficients must also be at most $\Vert \psi
\Vert$. From the coefficient of $\lambda^{k_\ell}$, the result follows.
\end{proof}

\subsection{Bounding a growing function}

We return to the setting of Proposition \ref{prop:boundLambdaf}.
Let $f = \sum_{n \geq 0} f_n$, $g = \sum_{n \geq 0} g_n$ and
$h = \sum_{n \geq 0} h_n$ be the analytic expansions of
$f$, $g$ and $h$.
Here $f_n$, $g_n$ and $h_n$ are the restrictions to the diagonal of 
some symmetric $n$-linear forms $F_n$, $G_n$ and $H_n$, respectively.
We recall that $\Lambda(f)$ 
is the Legendre transform of the Newton polygon $\NP(f)$ defined in Section \ref{sec:matrices}
\cite[Proposition 3.9]{caruso-roe-vaccon:14a}, and that $\alpha$ is a
real number such that $\Vert n! \Vert \geq e^{-\alpha n}$ for all positive integers $n$.

\begin{lem}
\label{lem:boundLambdaf}
We keep the above notation. 
If $(a,b)$ satisfies
$b \geq a + \Lambda(g)\big( \! \max(b, \, \Lambda(h) (a)) \big)$
then $b \geq \Lambda(f)(a - \alpha)$.
\end{lem}

\begin{proof}
We have $f' = \sum_{n \geq 0} f'_n$ where
$$f'_n : U \to \mathcal L(E,F), \quad
x \mapsto \big(h \mapsto n \cdot F_n(h, x, x, \ldots, x)\big).$$
Taking $h = x$, we find
$\Vert f'_n \Vert \geq \Vert n f_n \Vert = 
\Vert n \Vert \cdot \Vert f_n \Vert$.
Combining this with Proposition \ref{prop:boundhr}, we get
$$\Vert (r+1) f_{r+1} \Vert \leq
  \sup_{m, (n_i)} \Vert g_m \Vert \cdot 
  \prod_{i=1}^m \max(\Vert f_{n_i} \Vert, \Vert h_{n_i} \Vert)$$
for all nonnegative integers $r$,
where the supremum runs over all pairs $(m, (n_i))$ where $m$
is a nonnegative integer and $(n_i)_{1 \leq i \leq m}$ is a sequence of
length $m$ of nonnegative integers such that $n_1 + \ldots + n_m = r$.
We set $u_r = \Vert r! f_r \Vert$. Multiplying the above inequality by
$\Vert r! \Vert$, we obtain:
\begin{equation}
\label{eq:boundurrec}
u_{r+1} \leq
  \sup_{m, (n_i)} \Vert g_m \Vert \cdot 
  \prod_{i=1}^m \max(u_{n_i}, \Vert n_i! h_{n_i} \Vert)
\end{equation}
since the multinomial coefficient $\binom r {\!n_1 \, \cdots \, n_m\!}$
is an integer and hence has norm at most $1$.
We now pick two real numbers $a$ and $b$ satisfying the hypothesis of 
the Lemma. Set $d = \Lambda(h)(a)$. Going back to the definitions of 
$\Lambda (h)$ and Legendre transform, we get $\Vert h_n \Vert \leq e^{- 
a n + d}$ for all $n$.
Similarly, from our hypothesis on $(a,b)$, we find
$\Vert g_m \Vert \leq e^{-\!\max(b,d)\cdot m + b - a}$ for all $m$.
We are now ready to prove $u_r \leq e^{-ar + b}$ by induction on $r$.
When $r = 0$, it is obvious because $u_0$ vanishes. Otherwise, the
induction follows from
\begin{align*}
u_{r+1} 
& \leq \sup_{m, (n_i)}
    e^{ -\max(b,d)\cdot m + b - a + \sum_{i=1}^m (-a n_i + \max(b,d))} \\
& = e^{ b - a - a r } = e^{ -a (r+1) + b}.
\end{align*}
From the definition of $u_r$, we obtain 
$\Vert f_r \Vert \leq u_r \cdot \Vert r! \Vert^{-1} \leq e^{-(a - \alpha) 
r + b}$. Thus $b \geq \Lambda(f)(a - \alpha)$.
\end{proof}


We can now conclude the proof of Proposition \ref{prop:boundLambdaf} as 
follows. Given $a \in \R$ and $b = a + \Lambda_g \circ \Lambda_h(a)$, we 
have to prove that $\Lambda(f)(a-\alpha) \leq b$ provided that $b \leq 
\nu$. Thanks to Proposition \ref{prop:boundLambdaf}, it is enough to 
check that such pairs $(a,b)$ satisfy the hypothesis of Lemma
\ref{lem:boundLambdaf}. Clearly:
$b \geq a + \Lambda(g) \circ \Lambda(h)(a)$
since $\Lambda_g \geq \Lambda(g)$, $\Lambda_h \geq \Lambda(h)$ and
$\Lambda_g$ is nondecreasing. Furthermore, from $b \leq \nu$, we get
$\Lambda_g(b) = \min_{x \in \R} \Lambda_g(x) \leq \Lambda_g \circ 
\Lambda_h(a)$, from which we derive 
$a + \Lambda_g(b) \leq a + \Lambda_g \circ \Lambda_h(a) = b$.

{\raggedright
\bibliographystyle{plain}
\bibliography{Biblio,extras}
}

\end{document}